\documentclass[preprint,12pt]{elsarticle}

\usepackage[T1]{fontenc}
\usepackage{geometry}
\usepackage{graphicx}%
\usepackage{amsmath,amssymb,amsfonts}%
\usepackage{amsthm}%
\usepackage{mathrsfs}%
\usepackage[title]{appendix}%
\usepackage{xcolor}%
\usepackage{textcomp}%
\usepackage{manyfoot}%
\usepackage{booktabs}%
\usepackage{algorithm}%
\usepackage{algorithmicx}%
\usepackage{algpseudocode}%
\usepackage{listings}%
\usepackage{indentfirst}
\usepackage{lineno}
\usepackage{mathrsfs}
\usepackage{array}
\usepackage{float}
\usepackage{subfigure}
\usepackage[justification=centering]{caption}
\usepackage{verbatim}
\usepackage{setspace}
\usepackage{makecell}
\usepackage{textcomp}
\usepackage{multirow}%

\usepackage{bm}
\usepackage{url}
\usepackage{hyperref}

\newcommand{\tensor}[1]{\bm{\mathcal{#1}}}

\numberwithin{equation}{section}
\newtheorem{thm}{Theorem}

\newtheorem{lem}{Lemma}

\begin{document}

\begin{frontmatter}



\title{HOSCF: Efficient decoupling algorithms for finding the best rank-one approximation of higher-order tensors}


\author[mymainaddress,myfirstaddress]{Chuanfu Xiao}
	\ead{chuanfuxiao@pku.edu.cn}

 \author[mysecondaddress]{Zeyu Li}
    \ead{li\_zeyu@pku.edu.cn}

 \author[mymainaddress,myfirstaddress]{Chao Yang}
	\ead{chao\_yang@pku.edu.cn}

 \address[mymainaddress]{School of Mathematical Sciences, Peking University, Beijing 100871, China}
 \address[myfirstaddress]{PKU-Changsha Institute for Computing and Digital Economy, Changsha 410205, China}
 \address[mysecondaddress]{Yuanpei College, Peking University, Beijing 100871, China}



\begin{abstract}
Best rank-one approximation is one of the most fundamental tasks in tensor computation. In order to fully exploit modern multi-core parallel computers,
it is necessary to develop decoupling algorithms for computing the best rank-one approximation of higher-order tensors at large scales.
In this paper, we first build a bridge between the rank-one approximation of tensors and the eigenvector-dependent nonlinear eigenvalue problem (NEPv), and then develop an efficient decoupling algorithm, namely the higher-order self-consistent field (HOSCF) algorithm, inspired by the famous self-consistent field (SCF) iteration frequently used in computational chemistry.
The convergence theory of the HOSCF algorithm and an estimation of the convergence speed are further presented.
In addition, we propose an improved HOSCF (iHOSCF) algorithm that incorporates the Rayleigh quotient iteration, which can significantly accelerate the convergence of HOSCF. Numerical experiments show that the proposed algorithms can efficiently converge to the best rank-one approximation of both synthetic and real-world tensors and can scale with high parallel scalability on a modern parallel computer.

\end{abstract}




\end{frontmatter}




\section{Introduction}\label{sec:section1}
Tensor is a natural representation tool for high-dimensional data, such as time-varying data, numerical simulation data, hyperspectral images, etc. With the help of tensor decompositions, tensor has been playing a central role in many application fields \cite{Bro2006, Kolda2009, Kroonenberg2008, Khoromskij2012, cichocki2016tensor, cichocki2017tensor, Sidiropoulos2017}. For example, CANDECOMP/PARAFAC (CP) and Tucker decompositions, as high-dimensional generalizations of matrix singular value decomposition (SVD), are commonly used to exploit the hidden information in high-dimensional data analysis \cite{Hitchcock1927,Tucker1966,Levin1963,Kroonenberg1980,henrion1994n,kiers2001three}, while tensor train (TT) and hierarchical Tucker (HT) decompositions can help alleviate the curse of dimensionality that occurs in scientific computing \cite{Oseledets2009,Hackbusch2009,Grasedyck2010,Oseledets2011,Hackbusch2014}.

Due to the existence of noise in real-world data, the low-rank approximation problem corresponding to tensor decomposition is often considered in practical applications. Among them, finding the best rank-one approximation of higher-order tensors is one of the most fundamental tasks \cite{de2000best,zhang2001rank,hu2018convergence}, which has many applications in fields such as wireless communications, signal processing, and pattern recognition \cite{decurninge2020tensor,de1996independent,shi2013multi,shi2019rank}. Furthermore, algorithms for the rank-one approximation problem can also be used to compute the rank-$R$ approximation of the tensor via a greedy strategy, using methods such as the greedy rank-one update algorithm \cite{devore1996some,leibovici1997decomposition}. It has been proved that the approximation error decays exponentially with the increase of the rank-one term \cite{qi2011best}.

There are a series of algorithms for solving the best rank-one approximation of higher-order tensors, among which the higher-order power method (HOPM) \cite{de2000best} is one of the most popular choices. It was pointed out that the HOPM algorithm is equivalent to the alternating least squares (ALS) method, and its convergence has been proved \cite{uschmajew2015a} and further analyzed \cite{hu2018convergence}. Alternating singular value decomposition (ASVD) is another important algorithm for the rank-one approximation of higher-order tensors \cite{de2000best,friedland2013best}. Compared with the HOPM algorithm, the ASVD algorithm updates two factors simultaneously by calculating the largest singular pair of the intermediate matrix, which usually has a faster convergence speed  \cite{guan2018convergence}. Essentially, both HOPM and ASVD algorithms belong to the class of nonlinear block Gauss-Seidel coordinate descent methods, and the tensor-times-vector chain (TTVc) operation used to update the factors is their main computational cost. However, since the updates of these factors are interdependent within each iteration of HOPM and ASVD, the TTVc operations cannot be executed in parallel, which is the main reason for the low parallel efficiency of these algorithms \cite{tritsiklis1989comparison,ortega2000iterative,vrahatis2003linear,ahmadi2021parallel}.
Therefore, it is necessary to design an efficient algorithm that can decouple the updates of factors, i.e., a \textit{decoupling algorithm} that is more suitable for modern multi-core parallel computers.

A simple way to design decoupling algorithms is to convert the block Gauss-Seidel iterations in the HOPM and ASVD algorithms into the Jacobi-style iterative forms, which we refer to as Jacobi-HOPM and Jacobi-ASVD, respectively in this paper. Although the Jacobi-HOPM and Jacobi-ASVD can decouple the updates of factors during the iterations, we will show later in Section \ref{sec:section6} that they break the monotonic convergence property and thus often fail to converge to the best rank-one approximation, especially for higher-order tensors. Another decoupling algorithm is the generalized Rayleigh quotient iteration (GRQI) \cite{zhang2001rank}, which was proposed from the perspective of the Karush-Kuhn-Tucker (KKT) equation corresponding to the rank-one approximation of tensors.
Although GRQI is shown to be quadratically convergent \cite{zhang2001rank}, it generally does not converge to the best rank-one approximation without an appropriate initial guess, due largely to its relatively small region of convergence.
Overall, most existing decoupling algorithms for best rank-one approximation of higher-order tensors suffer from the issues of slow or non-convergence in practice.

In this paper, we first show that finding the best rank-one approximation of higher-order tensors is theoretically equivalent to solving the largest eigenpair of an eigenvector-dependent nonlinear eigenvalue problem (NEPv), a reformulation of the KKT equation. Based on this observation, we propose an efficient decoupling algorithm called the higher-order self-consistent field (HOSCF) for the best rank-one approximation of higher-order tensors, inspired by the famous self-consistent field (SCF) iteration frequently used in computational chemistry \cite{roothaan1951new,martin_2004,saad2010numerical}. Compared to existing decoupling algorithms, the proposed HOSCF algorithm can theoretically guarantee to converge to the best rank-one approximation.
To demonstrate the parallel efficiency of HOSCF, we also implement it on a modern parallel computer, and numerical results show that it can achieve high parallel scalability.
Moreover, combined with Rayleigh quotient iteration, an improved HOSCF (iHOSCF) algorithm is further proposed to accelerate the convergence of HOSCF, and its convergence speed is comparable to that of HOPM and ASVD.


The organization of the remainder of this paper is as follows. In Section \ref{sec:section2}, we introduce some basic notations of higher-order tensors used in this paper and the rank-one approximation problem. Section \ref{sec:section3} builds a bridge between the rank-one approximation of tensors and NEPv. The HOSCF algorithm is presented in Section \ref{sec:section4}. In Section \ref{sec:section5}, we establish the convergence theory of HOSCF and propose the iHOSCF algorithm combined with Rayleigh quotient iteration. Test results on several numerical experiments are reported in Section \ref{sec:section6}. The paper is concluded in Section \ref{sec:section7}.

\section{Preliminaries}\label{sec:section2}
In this section, we give some preliminaries about higher-order tensors. Let $\bm{\mathcal{A}}\in\mathbb{R}^{I_{1}\times I_{2}\cdots\times I_{d}}$
be a $d$th-order tensor, and $\bm{\mathcal{A}}_{i_{1},i_{2},\cdots,i_{d}}$ be its $(i_{1},i_{2},\cdots,i_{d})$-th element. We denote the Frobenius norm of $\tensor{A}$ as $\|\tensor{A}\|_{F}$, which is defined as $\sqrt{\sum\limits_{i_{1},i_{2},\cdots,i_{d}}\tensor{A}_{i_{1},i_{2},\cdots,i_{d}}^{2}}$.

The mode-$n$ matricization is an operation that reshapes the tensor $\tensor{A}$ into a matrix $\bm{A}_{(n)}\in\mathbb{R}^{I_{n}\times(I_{1}\cdots I_{n-1}I_{n+1}\cdots I_{d})}$ along the $n$ mode ($n\in\{1,2,\cdots,d\}$). Specifically, there is a mapping relationship between the $(i_{1},i_{2},\cdots,i_{d})$-th element of $\tensor{A}$ and the $(i_{n},j)$-th element of $\bm{A}_{(n)}$, where
\[
j = 1 + \sum\limits_{k=1,k\neq n}^{d}(i_{k}-1)J_{k}\ \ \text{with}\ \ J_{k}=\prod\limits_{l=1,l\neq n}^{k-1}I_{l}.
\]
Tensor-times-vector (TTV) is a basic operation in tensor computation. Suppose that $\bm{v}\in\mathbb{R}^{I_{n}}$ is a vector, the multiplication of $\tensor{A}$ and $\bm{v}$ along the mode-$n$ is denoted as $\tensor{A}\times_{n}\bm{v}$, which is a $(d-1)$th-order tensor in $\mathbb{R}^{I_{1}\cdots\times I_{n-1}\times I_{n+1}\cdots\times I_{d}}$. Elementwisely,
\[
[\tensor{A}\times_{n}\bm{v}]_{i_{1},\cdots i_{n-1},i_{n+1},\cdots, i_{d}} = \sum\limits_{i_{n}=1}^{I_{n}}\tensor{A}_{i_{1},\cdots, i_{n},\cdots, i_{d}}\cdot\bm{v}_{i_{n}}.
\]
Further, the TTV chain (TTVc) operation is defined as the multiplication of $\bm{\mathcal{A}}$ and multiple vectors along corresponding modes. For example, let
$\bm{v}^{\{n_{1},n_{2},\cdots,n_{K}\}}=\{\bm{v}^{(n_{k})}\in\mathbb{R}^{I_{n_{k}}}:k=1,2,\cdots,K\}$ be a set consisting of $K$ vectors, and $n_{k}\in\{1,2,\cdots,d\}$ for all $k=1,2,\cdots,K$, then $\bm{\mathcal{A}}\times_{n_{1}}\bm{v}^{(n_{1})}\times_{n_{2}}\bm{v}^{(n_{2})}\cdots\times_{n_K}\bm{v}^{(n_{K})}$ is a TTVc operation, which is also denoted as
$
\bm{\mathcal{A}}\times_{n_{1},n_{2},\cdots,n_{K}}\bm{v}^{\{n_{1},n_{2},\cdots,n_{K}\}}
$ in this paper.

Unlike the case of matrices, the definition of rank for higher-order tensors is not unique. Despite this, the various definitions of tensor rank, such as the CP-rank, the multilinear rank (i.e., Tucker-rank), and the TT-rank, all correspond to the same form in the rank-one case. Suppose that  $\tensor{A}$ is a rank-one tensor, then it can be represented as an outer product of $d$ vectors, i.e.,
\begin{equation}\label{eq:rank-one}
    \tensor{A} = \bm{u}^{(1)}\circ\bm{u}^{(2)}\cdots\circ\bm{u}^{(d)},
\end{equation}
where $\bm{u}^{(n)}\in\mathbb{R}^{I_{n}}$ for all $n=1,2,\cdots,d$. Specifically, the rank-one tensor $\bm{\mathcal{A}}$ has the entries
\[
\tensor{A}_{i_{1},i_{2},\cdots, i_{d}} = \bm{u}^{(1)}_{i_{1}}\bm{u}^{(2)}_{i_{2}}\cdots\bm{u}^{(d)}_{i_{d}}.
\]
In addition, based on the rank-one representation \eqref{eq:rank-one}, the corresponding basic operations on the tensor $\tensor{A}$ can be simplified. For example, the Frobenius norm of $\tensor{A}$ is equivalent to the product of the norm of these $d$ vectors, the mode-$n$ matricization of $\tensor{A}$ can be expressed as $\bm{u}^{(n)}\left(\bm{u}^{(d)}\cdots\otimes\bm{u}^{(n+1)}\otimes\bm{u}^{(n-1)}\cdots\otimes\bm{u}^{(d)}\right)^{T}$, and the multiplication of $\tensor{A}$ and $\bm{v}$ along the mode-$n$ can be reformulated by
\[
\tensor{A}\times_{n}\bm{v} = \left(\bm{v}^{T}\bm{u}^{(n)}\right)\cdot\bm{u}^{(1)}\cdots\circ\bm{u}^{(n-1)}\circ\bm{u}^{(n+1)}\cdots\circ\bm{u}^{(d)},
\]
which is still a rank-one tensor and only requires a vector inner product operation.

The rank-one approximation of higher-order tensors is one of the most fundamental problems in tensor computation.
Formally, for a $d$th-order tensor $\tensor{A}$,  the rank-one approximation problem can be written as
\begin{equation}\label{eq:rank-one approx-1}
    \min\limits_{\bm{u}^{(1)},\bm{u}^{(2)},\cdots,\bm{u}^{(d)}}\|\tensor{A} - \bm{u}^{(1)}\circ\bm{u}^{(2)}\cdots\circ\bm{u}^{(d)}\|_{F},
\end{equation}
which is a nonlinear and non-convex optimization problem, and solving it is NP-hard \cite{hillar2013most}. By normalizing the factor vectors, problem \eqref{eq:rank-one approx-1} is equivalent to the following spherical-constrained optimization problem
\begin{equation}\label{eq:rank-one approx-2}
\begin{split}
\max\limits\limits_{\bm{u}^{(1)},\bm{u}^{(2)},\cdots,\bm{u}^{(d)}}\tensor{A}\times_{1}\bm{u}^{(1)}\times_{2}\bm{u}^{(2)}\cdots\times_{d}\bm{u}^{(d)}\\
\text{s.t.}\ \ \|\bm{u}^{(n)}\|_{2} = 1\ \text{for all}\ n=1,2,\cdots,d.
\end{split}
\end{equation}
The KKT equation corresponding to optimization problem \eqref{eq:rank-one approx-2} is as follows
\begin{equation}\label{eq:kkt}
\left\{\begin{array}{c}
\tensor{A}\times_{2}\bm{u}^{(2)}\times_{3}\bm{u}^{(3)}\cdots\times_{d}\bm{u}^{(d)} = \lambda\bm{u}^{(1)},\\
\tensor{A}\times_{1}\bm{u}^{(1)}\times_{3}\bm{u}^{(3)}\cdots\times_{d}\bm{u}^{(d)}= \lambda\bm{u}^{(2)},\\
 \vdots \\    \tensor{A}\times_{1}\bm{u}^{(1)}\times_{2}\bm{u}^{(2)}\cdots\times_{d-1}\bm{u}^{(d-1)} = \lambda\bm{u}^{(d)},\\    \tensor{A}\times_{1}\bm{u}^{(1)}\times_{2}\bm{u}^{(2)}\cdots\times_{d}\bm{u}^{(d)} = \lambda.\\
 \end{array}
 \right.
\end{equation}
Generally, if $\left(\lambda_{*};\bm{u}^{(1)}_*,\bm{u}^{(2)}_*,\cdots,\bm{u}^{(d)}_*\right)$ satisfies the KKT equation \eqref{eq:kkt}, it is called a singular pair of the tensor $\tensor{A}$ \cite{lim2005singular}. Further, if $\lambda_{*}\cdot\bm{u}^{(1)}_*\circ\bm{u}^{(2)}_*\cdots\circ\bm{u}^{(d)}_*$ is the best rank-one approximation of $\tensor{A}$, then $|\lambda_{*}|$ is also called the spectral norm of $\tensor{A}$ \cite{hillar2013most}.

\section{A bridge between the rank-one approximation and NEPv}\label{sec:section3}

The eigenvector-dependent nonlinear eigenvalue problem (NEPv) aims to find $\bm{X}\in\mathbb{R}^{n\times r}$ with orthogonal columns and $\bm{\Lambda}\in\mathbb{R}^{r\times r}$ such that
\begin{equation}\label{eq:nepv}
    \bm{H}(\bm{X})\bm{X} = \bm{X}\bm{\Lambda},
\end{equation}
where $\bm{H}(\bm{X})\in\mathbb{R}^{n\times n}$ is a symmetric matrix-valued function of $\bm{X}$, and $\bm{\Lambda} = \bm{X}^{T}\bm{H}(\bm{X})\bm{X}$ and its eigenvalues are generally the
$r$ largest or smallest eigenvalues of $\bm{H}(\bm{X})$.
In this section, we present an equivalence theorem that shows that finding the best rank-one approximation of tensors is equivalent to solving the largest eigenpair of the NEPv, i.e., a reformulation of the KKT equation \eqref{eq:kkt}, which builds a bridge between the rank-one approximation of tensors and NEPv. Unlike the previous work \cite{bai2022variational}, our equivalence theorem is established based on the KKT equation \eqref{eq:kkt} corresponding to the rank-one approximation, and thus is applicable to tensors of any order.

For convenience, we denote $\mathbb{S}^{I-1}$ as  the unit sphere in $\mathbb{R}^{I}$ and $\text{Sym}\left(I\right)$ as the set of $I\times I$ real symmetric matrices.
Let $\tensor{A}\in\mathbb{R}^{I_1\times I_2\cdots\times I_d}$ be a $d$th-order tensor, and we define a mapping from $\mathbb{S}^{\sum\limits_{n=1}^{d}I_{n}-1}$ to $\text{Sym}\left(\sum\limits_{n=1}^{d}I_{n}\right)$ as follows,
\begin{equation}\label{eq:map}
    \mathcal{J}:\ \bm{x}\ \rightarrow\ \bm{J}(\bm{x})=\frac{1}{d-1}\left[\begin{array}{cccc}
        \bm{0} & \bm{A}_{1,2}(\bm{x}) & \cdots & \bm{A}_{1,d}(\bm{x})\\
        \bm{A}_{1,2}^{T}(\bm{x}) & \bm{0}& \cdots & \bm{A}_{2,d}(\bm{x}) \\
        \vdots & \vdots & \ddots & \vdots \\
        \bm{A}_{1,d}^{T}(\bm{x}) & \bm{A}_{2,d}^{T}(\bm{x})& \cdots & \bm{0}
    \end{array}\right],
\end{equation}
where $\bm{A}_{m,n}(\bm{x})\in\mathbb{R}^{I_{m}\times I_{n}}$ represents the matrix
\begin{equation}\label{eq:submatrix}
\tensor{A}\times_{1}\bm{u}^{(1)}\cdots\times_{m-1}\bm{u}^{(m-1)}\times_{m+1}\bm{u}^{(m+1)}\cdots\times_{n-1}\bm{u}^{(n-1)}\times_{n+1}\bm{u}^{(n+1)}\cdots\times_{d}\bm{u}^{(d)},
\end{equation}
and
\begin{equation}\label{eq:factor}
    \bm{u}^{(n)} =\left\{\begin{array}{cc}
    \frac{\bm{x}_{\texttt{index}_{n}}}{\|\bm{x}_{\texttt{index}_{n}}\|_{2}}, & \|\bm{x}_{\texttt{index}_{n}}\|_{2}\neq0 \\
    \bm{0}, & \|\bm{x}_{\texttt{index}_{n}}\|_{2}=0
\end{array}\right. \text{with}\  \texttt{index}_{n}=\sum\limits_{k=1}^{n-1}I_{k}+1:\sum\limits_{k=1}^{n}I_{k}
\end{equation}
for all $n=1,2,\cdots,d$. Based on the defined mapping $\mathcal{J}$, we can introduce the following NEPv,
\begin{equation}\label{eq:nepv-1}
    \bm{J}(\bm{x})\bm{x}=\lambda\bm{x},
\end{equation}
which is a special case of NEPv \eqref{eq:nepv}, i.e., $r=1$, and its properties are shown in Theorem \ref{thm:nepv-property}.

\begin{thm}\label{thm:nepv-property}
Let $\bm{x}\in\mathbb{S}^{\sum\limits_{n=1}^{d} I_n-1}$ be a unit vector,
the defined NEPv \eqref{eq:nepv-1} has the following two properties.\\
(P1) If $(\mu;\bm{y})$ is an eigenpair of $\bm{J}(\bm{x})$, then $(-\mu;\hat{\bm{y}})$ is also an eigenpair of $\bm{J}(\hat{\bm{x}})$, where
\begin{equation}\label{eq:symmetry}
        \hat{\bm{x}}_{\texttt{index}_n} = \left\{\begin{array}{cc}
        -\bm{x}_{\texttt{index}_n}, & n=1, \\
        \bm{x}_{\texttt{index}_n}, & n\neq1,
    \end{array}\right.
    \end{equation}
where $\texttt{index}_n$ is the same as Eq. \eqref{eq:factor}.
Further, $\mu$ is the largest magnitude eigenvalue of $\bm{J}(\bm{x})$ if and only if $\mu$ ($-\mu$) is the largest eigenvalue of $\bm{J}(\bm{x})$ ($\bm{J}(\hat{\bm{x}})$).\\
(P2) If $(\lambda;\bm{x})$ is a solution of NEPv \eqref{eq:nepv-1} and $\lambda\neq0$, then it satisfies
    \begin{equation}\label{eq:equivalent}
        \|\bm{x}_{\texttt{index}_{1}}\|_{2} = \|\bm{x}_{\texttt{index}_{2}}\|_{2}\cdots= \|\bm{x}_{\texttt{index}_{d}}\|_{2}.
    \end{equation}
\end{thm}

\begin{proof}
Let
\begin{equation*}
    \bm{v}^{(n)} =\left\{\begin{array}{cc}
    \frac{\bm{y}_{\texttt{index}_{n}}}{\|\bm{y}_{\texttt{index}_{n}}\|_{2}}, & \|\bm{y}_{\texttt{index}_{n}}\|_{2}\neq0, \\
    \bm{0}, & \|\bm{y}_{\texttt{index}_{n}}\|_{2}=0,
\end{array}\right.
\end{equation*}
for all $n=1,2,\cdots,d$, then by Eq. \eqref{eq:symmetry} and the definition of $\bm{J}(\bm{x})$, we have
    \begin{equation*}\label{eq:thm}
    \begin{split}
    \bm{J}(\hat{\bm{x}})\hat{\bm{y}}&=\frac{1}{(d-1)\sqrt{d}}\left[\begin{array}{cccc}
        \bm{0} & \bm{A}_{1,2}(\bm{x}) & \cdots & \bm{A}_{1,d}(\bm{x})\\
        \bm{A}_{1,2}^{T}(\bm{x}) & \bm{0}& \cdots & -\bm{A}_{2,d}(\bm{x}) \\
        \vdots & \vdots & \ddots & \vdots \\
        \bm{A}_{1,d}^{T}(\bm{x}) & -\bm{A}_{2,d}^{T}(\bm{x})& \cdots & \bm{0}
    \end{array}\right]\left[\begin{array}{c}
         -\bm{v}^{(1)}  \\
         \bm{v}^{(2)} \\
         \vdots \\
         \bm{v}^{(d)}
    \end{array}\right]\\
    &= \frac{1}{\sqrt{d}}\left[\begin{array}{c}
         \mu\bm{v}^{(1)}  \\
         -\mu\bm{v}^{(2)} \\
         \vdots \\
         -\mu\bm{v}^{(d)}
    \end{array}\right] = -\mu\hat{\bm{y}},
    \end{split}
    \end{equation*}
 which means that $(-\mu;\hat{\bm{y}})$ is also an eigenpair of NEPv \eqref{eq:nepv-1}.
 In other words, the eigenvalues of $\bm{J}(\bm{x})$ and $\bm{J}(\hat{\bm{x}})$ are opposite each other, that is,
 \[
 \lambda(\bm{J}(\hat{\bm{x}})) = -\lambda(\bm{J}(\bm{x})),
 \]
where $\lambda(\bm{J}(\bm{x}))$ represents the spectrum of $\bm{J}(\bm{x})$. It is easy to know that if $\mu$ is the largest magnitude eigenvalue of $\bm{J}(\bm{x})$, then it is either the largest eigenvalue of $\bm{J}(\bm{x})$ ($\mu>0$) or the largest eigenvalue of ${\bm{J}}(\hat{\bm{x}})$ ($\mu<0$).

   Let $\alpha_n=\|\bm{x}_{\texttt{index}_n}\|_{2}$ and $\beta_n=\frac{\sum\limits_{k\neq n}\alpha_k}{d-1}$ for all $n=1,2,\cdots,d$, then by the definition of  $\bm{J}(\bm{x})$,
   we can rewrite $\bm{J}(\bm{x})\bm{x}=\lambda\bm{x}$ as
\begin{equation}\label{eq:thm1-0}
        \left\{\begin{array}{c}
\beta_1\tensor{A}\times_{2}\bm{u}^{(2)}\times_{3}\bm{u}^{(3)}\cdots\times_{d}\bm{u}^{(d)} =\alpha_{1}\lambda\bm{u}^{(1)},\\
\beta_2\tensor{A}\times_{1}\bm{u}^{(1)}\times_{3}\bm{u}^{(3)}\cdots\times_{d}\bm{u}^{(d)} =\alpha_{2}\lambda\bm{u}^{(2)},\\
\vdots\\
\beta_{d}\tensor{A}\times_{1}\bm{u}^{(1)}\times_{2}\bm{u}^{(2)}\cdots\times_{d-1}\bm{u}^{(d-1)} =\alpha_{d}\lambda\bm{u}^{(d)},\\
 \end{array}
 \right.
    \end{equation}
    where
    \begin{equation*}
    \bm{u}^{(n)} =\left\{\begin{array}{cc}
    \frac{\bm{x}_{\texttt{index}_{n}}}{\|\bm{x}_{\texttt{index}_{n}}\|_{2}}, & \|\bm{x}_{\texttt{index}_{n}}\|_{2}\neq0, \\
    \bm{0}, & \|\bm{x}_{\texttt{index}_{n}}\|_{2}=0,
\end{array}\right. n=1,2,\cdots,d.
\end{equation*}
    By multiplying $\bm{u}^{(1)},\bm{u}^{(2)},\cdots,\bm{u}^{(d)}$ on each side of Eq. \eqref{eq:thm1-0}, we obtain
    \begin{equation}\label{eq:thm1-1}
    \beta_n\mu = \alpha_n\lambda,\ n=1,2\cdots,d,
    \end{equation}
where $\mu=\tensor{A}\times_1\bm{u}^{(1)}\times_{2}\bm{u}^{(2)}\cdots\times_{d}\bm{u}^{(d)}$.
Without loss of generality, we assume $\alpha_1=0$. Then by $\alpha_n\geq0$ and $\sum\limits_{n=1}^{d}\alpha_{n}^{2}=1$, we know that $\beta_1>0$ holds, which implies $\mu=0$. Since there is at least one $n\in\{2,\cdots,d\}$ such that $\alpha_n>0$, we can obtain $\lambda=0$, which leads to a contradiction.
Similarly, it is easy to prove that $\alpha_{n}>0$ holds for all $n=1,2,\cdots,d$.
Then by Eq. \eqref{eq:thm1-1}, we have
\begin{equation*}\label{eq:thm1-2}
\frac{\sum\limits_{k\neq 1}\alpha_{k}}{\alpha_{1}} = \frac{\sum\limits_{k\neq 2}\alpha_{k}}{\alpha_{2}}\cdots=\frac{\sum\limits_{k\neq d}\alpha_{k}}{\alpha_{d}},
\end{equation*}
which implies that
\[
\alpha_n\sum\limits_{k\neq m}\alpha_k=\alpha_m\sum\limits_{k\neq n}\alpha_k
\]
holds for all $m,n=1,2,\cdots,d$ and $m<n$, that is,
\[
\left(\alpha_{n}-\alpha_m\right)(\alpha_m+\alpha_n)=-(\alpha_n-\alpha_m)\sum\limits_{k\neq m,n}\alpha_k.
\]
If $\alpha_{m}\neq\alpha_n$, then we have
$\sum\limits_{n=1}^{d}\alpha_n=0,$
which contradicts $\alpha_n>0$.
Therefore, $\alpha_m=\alpha_n$ holds for all $m, n=1,2,\cdots,d$ and $m<n$.
\end{proof}

With the help of the properties of NEPv \eqref{eq:nepv-1} illustrated in Theorem \ref{thm:nepv-property}, we can now give an equivalence theorem, which illustrates that finding the best rank-one approximation of $\tensor{A}$ is equivalent to solving the largest eigenpair of NEPv \eqref{eq:nepv-1}, see Theorem \ref{thm:equivalent}.

\begin{thm}\label{thm:equivalent}
If $\lambda_{*}\cdot\bm{u}^{(1)}_*\circ\bm{u}^{(2)}_*\cdots\circ\bm{u}^{(d)}_*$ is the best rank-one approximation of $\tensor{A}$, then $(\lambda_*;\bm{x}_*)$ (or $(-\lambda_*;\hat{\bm{x}}_*)$) is also the largest eigenpair of NEPv \eqref{eq:nepv-1}, where
    \[
    \bm{x}_{*} = \frac{1}{\sqrt{d}}\left[\begin{array}{c}
     \bm{u}^{(1)}_* \\
      \bm{u}^{(2)}_*\\
      \vdots\\
      \bm{u}^{(d)}_*
\end{array}\right]\ \text{and}\ \hat{\bm{x}}_*=\frac{1}{\sqrt{d}}\left[\begin{array}{c}
     -\bm{u}^{(1)}_* \\
      \bm{u}^{(2)}_*\\
      \vdots\\
      \bm{u}^{(d)}_*
\end{array}\right] \in\mathbb{S}^{\sum\limits_{n=1}^{d}I_n-1}.
    \]
    Conversely, if
$(\lambda_{+};\bm{x}_{+})$ is the largest eigenpair of NEPv \eqref{eq:nepv-1}, then $\lambda_+\cdot\bm{u}_+^{(1)}\circ\bm{u}^{(2)}_+\cdots\circ\bm{u}_+^{(d)}$ is also the best rank-one approximation of $\tensor{A}$, where
\[
\bm{u}^{(n)}_+ =\left\{\begin{array}{cc}
    \frac{[\bm{x}_+]_{\texttt{index}_{n}}}{\|[\bm{x}_+]_{\texttt{index}_{n}}\|_{2}}, & \|[\bm{x}_+]_{\texttt{index}_{n}}\|_{2}\neq0 \\
    \bm{0}, & \|[\bm{x}_+]_{\texttt{index}_{n}}\|_{2}=0
\end{array}\right. \text{with}\  \texttt{index}_{n}=\sum\limits_{k=1}^{n-1}I_{k}+1:\sum\limits_{k=1}^{n}I_{k}
\]
for all $n=1,2,\cdots,d$.
\end{thm}

\begin{proof}
Without loss of generality, we assume $\lambda_*>0$. Since $\lambda_{*}\cdot\bm{u}^{(1)}_*\circ\bm{u}^{(2)}_*\cdots\circ\bm{u}^{(d)}_*$ is the best rank-one approximation of $\tensor{A}$, it satisfies the KKT equation \eqref{eq:kkt}, which can be rewritten as
\[
\bm{J}(\bm{x}_{*})\bm{x}_{*}=\lambda_{*}\bm{x}_{*}.
\]
Clearly, $(\lambda_*;\bm{x}_*)$ is also a solution of NEPv \eqref{eq:nepv-1}.
If there exists a solution $(\lambda_+;\bm{x}_+)$ of NEPv \eqref{eq:nepv-1} such that $\lambda_+>\lambda_*$, then by Theorem \ref{thm:nepv-property}, we have
\begin{equation*}\label{eq:equality}
    \|[\bm{x}_+]_{\texttt{index}_{1}}\|_{2} = \|[\bm{x}_+]_{\texttt{index}_{2}}\|_{2}\cdots= \|[\bm{x}_+]_{\texttt{index}_{d}}\|_{2},
\end{equation*}
and
\[
\lambda_+ = \tensor{A}\times_1\bm{u}_+^{(1)}\times_2\bm{u}_+^{(2)}\cdots\times_{d}\bm{u}^{(d)}_+,
\]
which contradicts $\lambda_+>\lambda_*$.
For the case of $\lambda_*<0$, we can similarly prove that $\lambda_+\leq-\lambda_*$ holds.
Therefore, if $\lambda_*>0$, then $(\lambda_*;\bm{x}_*)$ is the largest eigenpair of NEPv \eqref{eq:nepv-1}, otherwise $(-\lambda;\hat{\bm{x}}_*)$ is the largest eigenpair.
\end{proof}

In fact, for second-order tensors, i.e., matrices, the corresponding mapping $\mathcal{J}$ will degenerate to
\[
\mathcal{J}(\bm{x})=\bm{J}\equiv\left[\begin{array}{cc}
    \bm{0} & \bm{A} \\
    \bm{A}^{T} & \bm{0}
\end{array}\right],\ \forall \bm{x}\in\mathbb{S}^{I_1+I_2-1}.
\]
According to the Eckart-Young-Mirsky theorem \cite{eckart1936approximation,mirsky1960symmetric}, we know that the best rank-one approximation of $\bm{A}$ is equivalent to the largest singular pair, which is also the largest eigenpair of $\bm{J}$.
Therefore,
Theorem \ref{thm:equivalent} can be regarded as a generalization of the result that demonstrates the equivalence between finding the best rank-one approximation of $\bm{A}\in\mathbb{R}^{I_1\times I_2}$ and solving the largest eigenpair of the symmetric matrix $\bm{J}$.
Moreover,
we know that the largest eigenpair of $\bm{J}$ is also the largest magnitude eigenpair, Theorem \ref{thm:coro} shows that it can be also generalized to higher-order tensors.

\begin{thm}\label{thm:coro}
If $(\lambda_*;\bm{x}_*)$ is the largest eigenpair of NEPv \eqref{eq:nepv-1}, then it is also the largest magnitude eigenpair of $\bm{J}(\bm{x}_*)$.
\end{thm}

\begin{proof}
Since $(\lambda_*;\bm{x}_*)$ is the largest eigenpair of NEPv \eqref{eq:nepv-1}, Theorem \ref{thm:equivalent} show that $\lambda_*\cdot\bm{u}_*^{(1)}\circ\bm{u}^{(2)}_*\cdots\circ\bm{u}^{(d)}_*$ is the best rank-one approximation of $\tensor{A}$, where
\[
\bm{u}_*^{(n)} =\left\{\begin{array}{cc}
    \frac{[\bm{x}_*]_{\texttt{index}_{n}}}{\|[\bm{x}_*]_{\texttt{index}_{n}}\|_{2}}, & \|[\bm{x}_*]_{\texttt{index}_{n}}\|_{2}\neq0 \\
    \bm{0}, & \|[\bm{x}_*]_{\texttt{index}_{n}}\|_{2}=0
\end{array}\right. \text{with}\  \texttt{index}_{n}=\sum\limits_{k=1}^{n-1}I_{k}+1:\sum\limits_{k=1}^{n}I_{k}
\]
for all $n=1,2,\cdots,d$. If $\bm{J}(\bm{x}_*)$ has an eigenpair $(\mu;\bm{y})$ such that $\lambda_*<|\mu|$, then by the definition of $\bm{J}(\bm{x})$, we have
\begin{equation}\label{eq:coro-1}
    \sum\limits_{k=2}^{d}\bm{A}_{1,k}(\bm{x}_*)\bm{v}^{(k)} = \mu\bm{v}^{(1)},
\end{equation}
where
\[
\bm{v}^{(n)} =\left\{\begin{array}{cc}
    \frac{\bm{y}_{\texttt{index}_{n}}}{\|\bm{y}_{\texttt{index}_{n}}\|_{2}}, & \|\bm{y}_{\texttt{index}_{n}}\|_{2}\neq0 \\
    \bm{0}, & \|\bm{y}_{\texttt{index}_{n}}\|_{2}=0
\end{array}\right.
\]
for all $n=1,2,\cdots,d$.
By multiplying $\bm{v}^{(1)}$ on each side of Eq. \eqref{eq:coro-1}, we obtain
\[
\frac{1}{d-1}\sum\limits_{k=2}^{d}\lambda_{1,k} = \mu,
\]
where $\lambda_{1,k} = \tensor{A}\times_1\bm{v}^{(1)}\cdots\times_{k-1}\bm{u}_*^{(k-1)}\times_{k}\bm{v}^{(k)}\times_{k+1}\bm{u}^{(k+1)}_*\cdots\times_{d}\bm{u}_{*}^{(d)}$ for all $k=2,\cdots,d$. Further, according to the absolute value triangle inequality, we have
\[
|\mu|\leq\frac{1}{d-1}\sum\limits_{k=2}^{d}|\lambda_{1,k}|,
\]
which shows that there exists at least one $k\in\{2,\cdots,d\}$ such that $|\lambda_{1,k}|\geq|\mu|$. Clearly. it contradicts that $\lambda_*\cdot\bm{u}^{(1)}_*\circ\bm{u}^{(2)}_*\cdots\circ\bm{u}^{(d)}_*$ is the best rank-one approximation of $\tensor{A}$.
\end{proof}

\section{Higher-order self-consistent field (HOSCF)}\label{sec:section4}

From Theorem \ref{thm:equivalent}, we know that finding the best rank-one approximation of $\tensor{A}$ can be transformed into solving the largest eigenpair of NEPv \eqref{eq:nepv-1}, which enables us to develop algorithms from the perspective of NEPv. In this section, we present an efficient decoupling algorithm for solving the rank-one approximation problem inspired by the classic SCF iteration, namely the higher-order SCF (HOSCF) algorithm.

\subsection{The algorithm}
The SCF iteration is the most widely used method for NEPv \eqref{eq:nepv} \cite{roothaan1951new,martin_2004,saad2010numerical}, it updates the eigenvector matrix $\bm{X}_k$ associated with the $r$ largest/smallest eigenvalues of $\bm{H}(\bm{X}_{k-1})$, where $\bm{X}_{k-1}$ is the eigenvector matrix at the previous iteration step.
However, if the plain SCF iteration is directly used to solve the largest eigenpair of the defined NEPv \eqref{eq:nepv-1}, we are uncertain whether to employ $\bm{x}_{k-1}$ or $\hat{\bm{x}}_{k-1}$ to construct the intermediate symmetric matrix $\bm{J}_k$ and whether the obtained eigenvector is $\bm{x}_{k}$ or $-\bm{x}_{k}$ at the $k$th iteration, which may lead to the failure of convergence to the true solution.
To address this issue, at the $k$th iteration of HOSCF, we compute the largest magnitude eigenpair of $\bm{J}_{k-1}$ instead of the largest eigenpair, which is proven to be feasible in Theorem \ref{thm:nepv-property}. The detailed computational procedure of HOSCF is described in Algorithm \ref{algo:XXX}.

\begin{algorithm}
	\setstretch{1.0}
	\normalsize
	\caption{HOSCF: Higher-Order Self-Consistent Field.}
	\label{algo:XXX}
	\begin{algorithmic}[1]
		\Require Tensor $\bm{\mathcal{A}}\in\mathbb{R}^{I_{1}\times I_{2}\cdots\times I_{d}}$, initial guess $\bm{u}_{0}^{(1)},\bm{u}_{0}^{(2)},\cdots, \bm{u}_{0}^{(d)}$
		\Ensure The best rank-one approximation of $\tensor{A}$: $\lambda_k\cdot\bm{u}_{k}^{(1)}\circ\bm{u}_{k}^{(2)}\cdots\circ \bm{u}_{k}^{(d)}$
   	\State $\lambda_{0}\ \leftarrow\  \bm{\mathcal{A}}\times_{1}\bm{u}_{0}^{(1)}\times_{2}\bm{u}_{0}^{(2)}\cdots\times_{d}\bm{u}_{0}^{(d)}$
		\State $k\ \leftarrow\ 0$
		\While{not convergent}
    \State $\bm{J}_{k-1}\ \leftarrow$ the symmetric matrix $\bm{J}(\bm{x}_{k-1})$
         \State $(\lambda_{k};\bm{x}_{k})\ \leftarrow$ the largest magnitude eigenpair of $\bm{J}_{k-1}$
         \For{$n=1,2,\cdots,d$}
        \State $\bm{u}_{k}^{(n)}\ \leftarrow\ [\bm{x}_{k}]_{\texttt{index}_{n}}$ with $\texttt{index}_{n}=\sum\limits_{m=1}^{n-1}I_{m}+1:\sum\limits_{m=1}^{n}I_{m}$
        \State $\bm{u}_{k}^{(n)}\ \leftarrow\ \bm{u}_{k}^{(n)}/\|\bm{u}_{k}^{(n)}\|_{2}$
        \EndFor
		\State $k\ \leftarrow\ k+1$
		\EndWhile
        \If{$\lambda_{k}<0$}
        \State $\lambda_{k}\ \leftarrow\ -\lambda_{k}$
        \State $\bm{u}_{k}^{(1)}\ \leftarrow\ -\bm{u}_{k}^{(1)}$
        \EndIf
	\end{algorithmic}
\end{algorithm}

From Algorithm \ref{algo:XXX}, we know that these factors $\{\bm{u}_{k}^{(n)}:n=1,2,\cdots,d\}$ are updated simultaneously by solving the largest magnitude eigenpair of $\bm{J}_{k-1}$, depending only on factors from the previous iteration. Therefore, HOSCF is a decoupling algorithm, it has higher parallel efficiency than coupled algorithms such as HOPM and ASVD. On the other hand, compared with the GRQI algorithm,
the HOSCF algorithm takes into account the maximization requirement during the iterations, and thus usually has a relatively large convergence region.

\subsection{Computational complexity} In each iteration of HOSCF, the main computational procedure consists of two parts. The first part is the construction of the symmetric matrix $\bm{J}_{k-1}\in\text{Sym}\left(\sum\limits_{n=1}^{d}I_n\right)$ (i.e., line 4 of Algorithm \ref{algo:XXX}), whose time cost is $\mathcal{O}\left(d^2\prod\limits_{n=1}^{d}I_{n}\right)$. Clearly, the time cost for the construction of $\bm{J}_{k-1}$ grows exponentially with the order of $\tensor{A}$, it is the computational bottleneck of Algorithm \ref{algo:XXX}, especially when $d$ is large.
Fortunately, the calculations of the different blocks of $\bm{J}_{k-1}$ are independent, so we can improve its computational efficiency by a parallel implementation.
The second part is finding the largest magnitude eigenpair of $\bm{J}_{k-1}$ (i.e., line 5 of Algorithm \ref{algo:XXX}) with a time cost of $\mathcal{O}\left(\sum\limits_{n=1}^{d}I_{n}\right)$, which grows only linearly with the order of the tensor $\tensor{A}$.
In addition, some techniques in randomized numerical linear algebra can also be used to further accelerate the calculation of the largest magnitude eigenpair of the symmetric matrix $\bm{J}_{k-1}$, see \cite{Halko2011,Mahoney2016,martinsson2020randomized}.

\subsection{Stopping criteria} It is worth mentioning that a suitable stopping criteria is critical for Algorithm \ref{algo:XXX}. There are some stopping criteria commonly used in practice, such as
\[
|\lambda_{k+1} - \lambda_{k}|\leq \texttt{tol}\cdot|\lambda_{k}|
\]
and
\[
|\sin\theta(\bm{x}_{k+1},\bm{x}_{k})|\leq \texttt{tol},
\]
where $\texttt{tol}$ is a predetermined error tolerance, and $\theta(\bm{x}_{k+1},\bm{x}_{k})$ represents the angle between $\bm{x}_{k+1}$ and $\bm{x}_{k}$. However, the above two stopping criteria are based on the change between adjacent iteration steps, which generally cannot reflect whether the HOSCF algorithm converges. To this end, we refer to \cite{cai2018eigenvector} to give the following stopping criteria based on residual, i.e.,
\begin{equation}\label{eq:stop-3}
\frac{\|\bm{J}_{k-1}\bm{x}_{k} -\rho_{k}\bm{x}_{k}\|_{2}}{\|\bm{J}_{k-1}\|_{F}+|\lambda_{k}|}\leq \texttt{tol},
\end{equation}
where $\rho_{k}=\bm{x}_{k}^{T}\bm{J}_{k-1}\bm{x}_{k}$ $\left(\|\bm{x}_{k}\|_{2}=1\right)$ is called the Rayleigh quotient. Obviously, the stopping criteria \eqref{eq:stop-3} reflects the satisfaction of the KKT equation \eqref{eq:kkt}, which is more reasonable in practice.

\subsection{Improved HOSCF (iHOSCF)} Furthermore, to improve the convergence speed of HOSCF, we refer to \cite{bai2022variational} to introduce Rayleigh quotient iteration during the iterations of HOSCF. Specifically, after line 5 of Algorithm \ref{algo:XXX} is executed, we run one step of Rayleigh quotient iteration, that is,
\[
\begin{gathered}
\text{calculation of the Rayleigh quotient:}\ \rho_k\leftarrow\frac{\bm{x}_{k}^{T}\bm{J}_{k-1}\bm{x}_{k}}{\bm{x}_{k}^{T}\bm{x}_{k}},\\
\text{Rayleigh quotient iteration:}\ \bm{x}_{k}\leftarrow\left(\bm{J}_{k-1}-\rho_{k}\bm{I}\right)^{-1}\bm{x}_{k},\\
\text{normalization:}\ \bm{x}_{k}\leftarrow\frac{\bm{x}_{k}}{\|\bm{x}_{k}\|_{2}}.
\end{gathered}
\]
In practice, the updated $\bm{x}_{k}$ by the Rayleigh quotient iteration is only accepted in the proposed improved HOSCF (iHOSCF) if the corresponding eigenvalue $\lambda_k$ increases, which is to eliminate counteractions caused by the fact that $\bm{x}_k$ is far from the exact solution $\bm{x}_{*}$ at the beginning of the iteration.
Thanks to the cubic convergence property of the Rayleigh quotient iteration \cite{demmel1997applied}, iHOSCF significantly reduces the number of iterations of HOSCF and achieves a comparable convergence speed to ASVD.

\section{Convergence analysis}\label{sec:section5}
In this section, we will establish the convergence theory of the proposed HOSCF algorithm.
Before that, we need the following two lemmas. The first one is the famous perturbation theory on symmetric matrix, and the second one provides the upper bounds of $\|\bm{J}-\bm{J}_*\|_2$, $\|(\bm{J}-\bm{J}_*)\bm{x}_*\|_2$, and $\bm{x}_*^{T}(\bm{J}-\bm{J}_*)\bm{x}_*$ respectively.

\begin{lem}\label{lem:perturbation-bound}
    (See \cite{sun1987matrix,stewart1990matrix,li2006matrix}.) Let $\bm{A},\ \hat{\bm{A}}=\bm{A}+\bm{E}\in\mathbb{R}^{n\times n}$ be two real symmetric matrices, and
    \[
    \lambda_{1}\geq\lambda_2\cdots\geq\lambda_n\ \text{and}\ \hat{\lambda}_1\geq\hat{\lambda}_2\cdots\geq\hat{\lambda}_n
    \]
    are their eigenvalues. Assume that the gap $\gamma$ between the $r$ and $(r+1)$th eigenvalues is larger than 0, then the following inequalities
    \begin{equation*}\label{eq:perturbation-bound1}
    \begin{split}
        |\hat{\lambda}_r - \lambda_r| &\leq\|\bm{E}\|_2\\
        \end{split}
        \end{equation*}
        and
        \begin{equation*}
            \begin{split}
        \|\sin\angle\bm{\Theta}(\bm{V},\hat{\bm{V}})\|_{2} &\leq \frac{\|(\bm{I}-\hat{\bm{P}})\bm{E}\bm{P}\|_2}{\hat{\lambda}_r - \lambda_{r+1}}\\
    \end{split}
    \end{equation*}
    hold,
    where $\bm{V},\hat{\bm{V}}\in\mathbb{R}^{n\times r}$ are matrices associated with the $r$ largest eigenvalues of $\bm{A}$ and $\hat{\bm{A}}$, and $\bm{P} = \bm{V}\bm{V}^{T}$ and $\hat{\bm{P}}=\hat{\bm{V}}\hat{\bm{V}}^{T}$ represent orthogonal projection matrices.
\end{lem}

\begin{lem}\label{lem:J-bound}
    Let $(\lambda_*;  \bm{x}_*)$ be the largest eigenpair of NEPv \eqref{eq:nepv-1},
    and
    $\bm{x}\in\mathbb{S}^{\sum\limits_{n=1}^{d}I_n-1}$ be a unit vector that is sufficiently close to $\bm{x}_*$. Suppose that
\[
\bm{x}_*=\frac{1}{\sqrt{d}}\left[\begin{array}{c}
         \bm{u}^{(1)}_* \\
         \bm{u}^{(2)}_* \\
         \vdots \\
         \bm{u}^{(d)}_*
    \end{array}\right]\ \text{and}\ \bm{x}=\frac{1}{\sqrt{d}}\left[\begin{array}{c}
         \bm{u}^{(1)} \\
         \bm{u}^{(2)} \\
         \vdots \\
         \bm{u}^{(d)}
    \end{array}\right],
\]
and $\lambda_*$ is a single eigenvalue of $\bm{J}_*$,
where
    $\bm{u}^{(n)} = \alpha_n\bm{u}_*^{(n)}+\beta_n\bm{v}^{(n)}$ ($\bm{v}^{(n)}\bot\bm{u}_*^{(n)}$)
    and $\alpha_n>0$, $\beta_n$ is a sufficiently small real number for all $n=1,2,\cdots,d$. Then there exists positive constants $C_1(d)$, $C_2(d)$, and $\varepsilon>0$ such that
    \begin{equation}\label{eq:bound1}
        \begin{split}
            \|\bm{J} - \bm{J}_*\|_2 &\leq  C_1(d)\beta+O(\beta^{1+\varepsilon}), \\
            \|(\bm{J}-\bm{J}_*)\bm{x}_*\|_2 &\leq  C_2(d)\beta+O(\beta^{1+\varepsilon}), \\
        \end{split}
    \end{equation}
    and
    \begin{equation}\label{eq:bound2}
        \bm{x}_*^T(\bm{J}-\bm{J}_*)\bm{x}_* \leq O(\beta^{1+\varepsilon})
    \end{equation}
    hold,
    where $\beta = \sqrt{\frac{\sum\limits_{n=1}^d\beta_n^2}{d}}$.
\end{lem}

\begin{proof}
Let $\varepsilon$ be a constant in $(0,1)$,
we first prove that
\begin{equation}\label{eq:0000}
    1-\alpha_n\leq O(\beta^{1+\varepsilon})
\end{equation}
holds for all $n=1,2,\cdots,d$. From $\alpha_n^2+\beta_n^2=1$, it is easy to know that
    \begin{equation}\label{eq:bound-00}
        1-\alpha_n = 1-\sqrt{1-\beta_n^2} = \frac{\beta_n^2}{1+\sqrt{1-\beta_n^2}}=\mathcal{O}(\beta_n^2),
    \end{equation}
then we have
    \[
    \frac{\beta_{n}^2}{\beta^{1+\varepsilon}}=\frac{d^{\frac{1+\varepsilon}{2}}\beta_n^{2}}{(\sum\limits_{n=1}^d\beta_n^2)^{\frac{1+\varepsilon}{2}}}\leq\frac{d^{\frac{1+\varepsilon}{2}}\beta_n^2}{\beta_n^{1+\varepsilon}},
    \]
    and
     \[
    \lim\limits_{\beta\rightarrow0}\frac{\beta_n^2}{\beta^{1+\varepsilon}}=0,
    \]
    which implies that $1-\alpha_n\leq O(\beta^{1+\varepsilon})$.

    For $\|\bm{J}-\bm{J}_*\|_2$, by the definition of $\bm{J}(\bm{x})$, we have
\begin{equation}\label{eq:bound-01}
    \begin{split}
        \bm{J}-\bm{J}_* & =  \frac{1}{d-1}\left[\begin{array}{cccc}
       \bm{0}  & \bm{A}_{1,2}-\bm{A}_{*1,2} & \cdots & \bm{A}_{1,d}-\bm{A}_{*1,d}   \\
        \bm{A}_{1,2}^T-\bm{A}_{*1,2}^{T} & \bm{0} & \cdots & \bm{A}_{2,d}-\bm{A}_{*2,d} \\
        \vdots & \vdots & \ddots & \vdots\\
         \bm{A}_{1,d}^T-\bm{A}_{*1,d}^{T} & \bm{A}_{2,d}^T-\bm{A}_{*2,d}^{T} & \cdots & \bm{0}
    \end{array}\right]. \\
    \end{split}
\end{equation}
Then by Eqs. \eqref{eq:submatrix}, \eqref{eq:0000}, and \eqref{eq:bound-00}, we can obtain
\begin{equation}\label{eq:bound-02}
    \begin{split}
        \|\bm{A}_{m,n} - \bm{A}_{*m,n}\|_2 & \leq \sum\limits_{k\neq m,n}\beta_k\|\tensor{A}_{*m,n,k}\times_k\bm{v}^{(k)}\|_2+O(\beta^{1+\varepsilon}),\ \forall m\neq n,
    \end{split}
\end{equation}
where
\[
\tensor{A}_{*m,n,k} = \tensor{A}\times_{\mathcal{I}}\bm{v}^{\mathcal{I}}\ \text{with}\ \mathcal{I}=\{1,2,\cdots,d\}/\{m,n,k\}.
\]
Since
\begin{equation*}\label{eq:bound-03}
\begin{split}
\|\tensor{A}_{*m,n,k}\times_k\bm{v}^{(k)}\|_2
&=\max\limits_{\bm{w}^{(m)}\in\mathbb{S}^{I_m-1},\bm{w}^{(n)}\in\mathbb{S}^{I_n-1}}\sum\limits_{k\neq m,n}\beta_k\tensor{A}_{*m,n,k}\times_m\bm{w}^{(m)}\times_{n}\bm{w}^{(n)}\times_k\bm{v}^{(k)},
\end{split}
\end{equation*}
it follows that
\begin{equation}\label{eq:bound-04}
\begin{split}
    \|\bm{A}_{m,n} - \bm{A}_{*m,n}\|_2&\leq\lambda_*\sum\limits_{k\neq m,n}|\beta_k|+O(\beta^{1+\varepsilon})\\
    &\leq\lambda_*(d-2)\sqrt{\sum\limits_{k\neq m,n}\beta_k^2}+O(\beta^{1+\varepsilon})\leq\lambda_*(d-2)\sqrt{d}\beta+O(\beta^{1+\varepsilon}).
\end{split}
\end{equation}
Furthermore, we know
\begin{equation*}\label{eq:bound-05}
    \begin{split}
        \|\bm{J}-\bm{J}_*\|_2&=\max\limits_{\bm{y}\in\mathbb{S}^{\sum\limits_{n=1}^dI_n-1}}\bm{y}^{T}(\bm{J}-\bm{J}_*)\bm{y},
    \end{split}
\end{equation*}
and combined with Eqs. \eqref{eq:bound-01} and \eqref{eq:bound-04}, we can easy to prove that $\|\bm{J}-\bm{J}_*\|_2$ is bounded by
\[
\frac{\lambda_*(d-1)(d-2)^2\sqrt{d}\beta}{2}+O(\beta^{1+\varepsilon}).
\]

We then consider $\|(\bm{J}-\bm{J}_*)\bm{x}_*\|_2$.
From Eqs. \eqref{eq:bound-01} and \eqref{eq:bound-02}, we have
\begin{equation}\label{eq:bound-10}
\begin{split}
    (\bm{J}-\bm{J}_*)\bm{x}_* &= \frac{1}{(d-1)\sqrt{d}}\left[\begin{array}{cccc}
       \bm{0}  & \bm{A}_{1,2}-\bm{A}_{*1,2} & \cdots & \bm{A}_{1,d}-\bm{A}_{*1,d}   \\
        \bm{A}_{1,2}^T-\bm{A}_{*1,2}^T & \bm{0} & \cdots & \bm{A}_{2,d}-\bm{A}_{*2,d} \\
        \vdots & \vdots & \ddots & \vdots\\
         \bm{A}_{1,d}^T-\bm{A}_{*1,d}^{T} & \bm{A}_{2,d}^T-\bm{A}_{*2,d}^{T} & \cdots & \bm{0}
    \end{array}\right]
    \left[\begin{array}{c}
          \bm{u}^{(1)}_* \\
         \bm{u}^{(2)}_* \\
         \vdots \\
         \bm{u}^{(d)}_*\\
    \end{array}\right] \\
    & \leq  \frac{d-2}{(d-1)\sqrt{d}}\left[\begin{array}{c}
          \sum\limits_{k\neq1}\beta_k\bm{A}_{*1,k}\bm{v}^{(k)} \\
          \sum\limits_{k\neq2}\beta_k\bm{A}_{*2,k}\bm{v}^{(k)} \\
         \vdots  \\
          \sum\limits_{k\neq d}\beta_k\bm{A}_{*d,k}\bm{v}^{(k)} \\
    \end{array}\right] + O(\beta^{1+\varepsilon}).\\
\end{split}
\end{equation}
Further, by the triangle inequality and Cauchy's inequality, we can give an upper bound of $\|(\bm{J}-\bm{J}_*)\bm{x}_*\|_2$ as follows
\begin{equation*}\label{eq:bound-11}
    \frac{(d-2)\beta}{d-1}\sqrt{2\sum\limits_{m<n}\sigma_{m,n}^2}+O(\beta^{1+\varepsilon}),
\end{equation*}
where $\sigma_{m,n}$ is the second largest singular value of the matrix $\bm{A}_{*m,n}$.

For $\bm{x}_*^T(\bm{J}-\bm{J}_*)\bm{x}_*$, by Eq. \eqref{eq:bound-10}, it is easy to check that
\begin{equation*}\label{eq:bound-20}
    \bm{x}_*^T(\bm{J}-\bm{J}_*)\bm{x}_*\leq\frac{d-2}{(d-1)d}\sum\limits_{n=1}^d\sum\limits_{k\neq n}\beta_k\bm{u}_*^{(n)}\bm{A}_{*n,k}\bm{v}^{(k)}+O(\beta^{1+\varepsilon}).
\end{equation*}
And from Theorem \ref{thm:equivalent}, we know that $\lambda_*$ is also the largest singular value of $\bm{A}_{*n,k}$ for all $n,k=1,2,\cdots,d$ and $k\neq n$. Since $\lambda_*$ is single and $\bm{v}^{(n)}\perp\bm{u}_*^{(n)}$, it is easy to know that $\bm{u}_*^{(n)}\bm{A}_{*n,k}\bm{v}^{(k)}=0$, which implies that Eq. \eqref{eq:bound2} holds.
\end{proof}

Now we are ready to give the convergence theory of the proposed HOSCF algorithm in Theorem \ref{thm:convergence}, which illustrates that
HOSCF is locally $q$-linearly convergent and gives an estimate of the convergence rate.

\begin{thm}\label{thm:convergence}
    Let $(\lambda_*;\bm{x}_*)$ be defined as the same as Lemma \ref{lem:J-bound}, and $\{\bm{x}_k\}$ be the sequence generated by HOSCF with initial guess $\bm{x}_0$.
    If the gap $\gamma_*$ between $\lambda_*$ and the second largest eigenvalue of $\bm{J}_*$ satisfies
    $\frac{C_2(d)}{\gamma_*}<1$,
    and $\bm{x}_0$ is sufficiently close to $\bm{x}_*$, then there exists a sequence $\{\tau_k\}$ such that
    \begin{equation}\label{eq:convergence-0}
        \beta^{k+1}\leq\tau_k\beta^k
    \end{equation}
    and
    \begin{equation}\label{eq:convergence-1}
        \lim\limits_{k\rightarrow+\infty}\tau_k=\frac{C_2(d)}{\gamma_*},
    \end{equation}
    where $\beta^k=\sqrt{\frac{\sum\limits_{n=1}^d(\beta_n^{k})^2}{d}}$ and
    $C_2(d)=\frac{(d-2)\sqrt{2\sum\limits_{m<n}\sigma_{m,n}^2}}{d-1}$ are the same as Lemma \ref{lem:J-bound}.
\end{thm}

\begin{proof}
    According to Lemma \ref{lem:perturbation-bound}, we can obtain
    \begin{equation}\label{eq:convergence-00}
        \begin{split}
            |\sin\angle\theta(\bm{x}_{k+1},\bm{x}_*)|&\leq\frac{\|(\bm{I}-\bm{P}_{k+1})(\bm{J}_k-\bm{J}_*)\bm{P}_*\|_2}{\gamma_*-\|\bm{J}_k-\bm{J}_*\|_2}\\
            &\leq\frac{\|(\bm{I}-\bm{P}_*)(\bm{J}_k-\bm{J}_*)\bm{P}_*\|_2}{\gamma_*-\|\bm{J}_k-\bm{J}_*\|_2} + \frac{|\sin\angle\theta(\bm{x}_{k+1},\bm{x}_*)|\|\bm{J}_k - \bm{J}_*\|_2}{\gamma_*-\|\bm{J}_k-\bm{J}_*\|_2},
        \end{split}
    \end{equation}
    it follows that
    \begin{equation*}\label{eq:convergence-01}
    \begin{split}
        |\sin\angle\theta(\bm{x}_{k+1},\bm{x}_*)| &\leq \frac{\|(\bm{I}-\bm{P}_*)(\bm{J}_k-\bm{J}_*)\bm{P}_*\|_2}{\gamma_*-2\|\bm{J}_k-\bm{J}_*\|_2}\\
        &\leq\frac{\|(\bm{I}-\bm{P}_*)(\bm{J}_k-\bm{J}_*)\bm{P}_*\|_F}{\gamma_*-2\|\bm{J}_k-\bm{J}_*\|_2}, \\
    \end{split}
    \end{equation*}
    where $\bm{P}_k = \bm{x}_k\bm{x}_k^T$ and $\bm{P}_*=\bm{x}_*\bm{x}_*^T$. Since
    \[
    |\sin\angle\theta(\bm{x}_{k+1},\bm{x}_*)| = \sqrt{1-\cos^2\angle\theta(\bm{x}_{k+1},\bm{x}_*)},
    \]
    we have
    \begin{equation}\label{eq:convergence-02}
    \begin{split}
         |\sin\angle\theta(\bm{x}_{k+1},\bm{x}_*)| & = \sqrt{1 -\left(\sum\limits_{n=1}^d\frac{\alpha^{k+1}_n}{d}\right)^2} \geq\sqrt{\frac{\sum\limits_{n=1}^d\left(1-\left(\alpha_n^{k+1}\right)^2\right)}{d}}\\
         & =  \sqrt{\frac{\sum\limits_{n=1}^{d}\left(\beta^{k+1}_n\right)^2}{d}} =\beta^{k+1}.
    \end{split}
    \end{equation}
    We then give an upper bound of $\|(\bm{I}-\bm{P}_*)(\bm{J}_k-\bm{J}_*)\bm{P}_*\|_F$ by Lemma \ref{lem:J-bound}. According to the property of the Frobenius norm,
    \begin{equation}\label{eq:convergence-03}
        \begin{split}
             \|(\bm{I}-\bm{P}_{*})(\bm{J}_k-\bm{J}_*)\bm{P}_{*}\|_{F}^2 & = \text{tr}((\bm{I}-\bm{P}_*)(\bm{J}_k-\bm{J}_*)\bm{P}_*^2(\bm{J}_k-\bm{J}_*)(\bm{I}-\bm{P}_*))\\
             & = \text{tr}((\bm{J}_k-\bm{J}_*)\bm{P}_*^2(\bm{J}_k-\bm{J}_*)(\bm{I}-\bm{P}_*)^2)\\
             & = \text{tr}((\bm{J}_k-\bm{J}_*)\bm{P}_*(\bm{J}_k-\bm{J}_*)(\bm{I}-\bm{P}_*))\\
             & = \text{tr}((\bm{J}_k-\bm{J}_*)\bm{P}_*(\bm{J}_k-\bm{J}_*))-\text{tr}((\bm{J}-\bm{J}_*)\bm{P}_*(\bm{J}-\bm{J}_*)\bm{P}_*),\\
        \end{split}
    \end{equation}
    where $\text{tr}(\cdot)$ denotes the trace of a matrix. Since the trace operation satisfies commutative law, we can further obtain
    \begin{equation}\label{eq:convergence-04}
        \begin{split}
            \text{tr}\left(\left(\bm{J}_k-\bm{J}_*\right)\bm{P}_*\left(\bm{J}_k-\bm{J}_*\right)\right) & =  \bm{x}_*^T\left(\bm{J}_k-\bm{J}_*\right)^2\bm{x}_* = \|\left(\bm{J}_k-\bm{J}_*\right)\bm{x}_*\|_2^2, \\
        \end{split}
    \end{equation}
    and
    \begin{equation}\label{eq:convergence-05}
        \begin{split}
            \text{tr}\left(\left(\bm{J}_k-\bm{J}_*\right)\bm{P}_*\left(\bm{J}_k-\bm{J}_*\right)\bm{P}_*\right) & = \left(\bm{x}_*^T\left(\bm{J}_k-\bm{J}_*\right)\bm{x}_*\right)^2.
        \end{split}
    \end{equation}
    Combined with Eqs. \eqref{eq:convergence-03}, \eqref{eq:convergence-04}, and \eqref{eq:convergence-05}, we have
    \begin{equation}\label{eq:convergence-06}
    \begin{split}
         \|(\bm{I}-\bm{P}_{*})\left(\bm{J}_k-\bm{J}_*\right)\bm{P}_{*}\|_{F}&\leq \sqrt{\|\left(\bm{J}_k-\bm{J}_*\right)\bm{x}_*\|_2^2 + \left(\bm{x}_*^T\left(\bm{J}_k-\bm{J}_*\right)\bm{x}_*\right)^2} \\
         & \leq\|\left(\bm{J}_k-\bm{J}_*\right)\bm{x}_*\|_2 + |\bm{x}_*^T\left(\bm{J}_k-\bm{J}_*\right)\bm{x}_*| \\
         &\leq C_2(d)\beta^k+O\left((\beta^k)^{1+\varepsilon}\right).
    \end{split}
    \end{equation}
   Further, by Eqs. \eqref{eq:convergence-00}, and \eqref{eq:convergence-02}, and \eqref{eq:convergence-06}, we can easy to obtain that
    \begin{equation}\label{eq:convergence-07}
    \begin{split}
        \beta^{k+1}&\leq\frac{C_2(d)\beta^k}{\gamma_*-2C_1(d)\beta_k}+O\left((\beta^k)^{1+\varepsilon}\right)\\
        &=\frac{C_2(d)}{\gamma_*}\beta^k+O\left((\beta^k)^{1+\varepsilon}\right)=\left(\frac{C_2(d)}{\gamma_*}+O\left((\beta^{k})^{\varepsilon}\right)\right)\beta^k.
        \end{split}
    \end{equation}
Since $\frac{C_2(d)}{\gamma_*}<1$ and $\beta^0$ is sufficiently small, and let $\tau_k = \frac{C_2(d)}{\gamma_*}+O((\beta^{k})^{\varepsilon})$, it is easy to know that
\[
\lim\limits_{k\rightarrow\infty}\tau_k=\frac{C_2(d)}{\gamma_*},
\]
which means that Eqs. \eqref{eq:convergence-0} and \eqref{eq:convergence-1} hold.
\end{proof}

\section{Numerical experiments}\label{sec:section6}
In this section, we will demonstrate the advantages of the proposed HOSCF algorithm and its improved version through numerical experiments, including two parts.
The first part compares the convergence of the HOSCF and iHOSCF algorithms with other decoupling algorithms, including Jacobi-HOPM, Jacobi-ASVD, and GRQI, with several examples. We carry out these algorithms in MATLAB R2019b, and the implementations of all the tested algorithms are based on Tensor Toolbox v3.1 \cite{Bader}.
The second part aims to test the parallel scalability of the HOSCF algorithm. We implement the HOSCF algorithm in C++ language with OpenMP multi-threading \cite{chandra2001parallel} and MPI \cite{graham2006open} multi-processes, and the involved linear algebra operations such as eigensolver available from the Intel MKL library \cite{wang2014intel}. The numerical experiment is conducted on a CPU cluster, each node equipped with an Intel Xeon Platinum 8358P CPU of 2.60 GHz.

\subsection{Convergence test}
In the first part of the numerical experiments, we compare the convergence of HOSCF and iHOSCF algorithms with other decoupling algorithms.
To be fair, these algorithms update each factor vector only once in one iteration, and the initial factor vectors $\{\bm{u}^{(n)}_{0}:n=1,2\cdots,d\}$ are generated by a uniform distribution of interval $[0,1]$. We refer to \cite{cai2018eigenvector} and use \eqref{eq:stop-3} as the stopping criteria. Unless mentioned otherwise, the error tolerance is set to $\texttt{tol}=1.0\times 10^{-4}$, and the maximum number of iterations is limited to $500$.
As a reference, we use the results obtained by the HOPM algorithm as a baseline. Specifically, we run the function $\texttt{tucker\_als}$ in Tensor Toolbox for 500 iterations, again using the same randomly generated initial guess.

\subsubsection{Example 1}
The input tensor $\tensor{A}\in\mathbb{R}^{I\times I\cdots\times I}$ is randomly generated, and each element of $\tensor{A}$ follows the Gaussian distribution $\mathcal{N}(0,1)$. We set the tensor size $I$ to $10$, and increase the order of $\tensor{A}$ from $d=3$ to $6$.
Table \ref{table:lambda} and Figure \ref{fig:iterations} show the average value of $\lambda$ and the average number of iterations for all tested algorithms with 50 initial guesses.

\begin{table}
	\renewcommand\arraystretch{1.0}
	\begin{center}
		\caption{\normalsize Values of $\lambda$ by all tested algorithms for Example 1.}\label{table:lambda}
		\begin{tabular}{c|c|c|c|c}
			\toprule
			Algorithms & $d=3$ & $d=4$ & $d=5$ & $d=6$ \\
			\midrule
			$\texttt{tucker\_als}$ & 8.26 $\pm$ 0.48 & 9.25 $\pm$0.43 & 11.13 $\pm$0.39 & 12.49 $\pm$ 0.35 \\

            \midrule

            Jacobi-HOPM &  8.25 $\pm$ 0.46 & 9.75 $\pm$ 0.79 & 10.20$\pm$ 2.59 & 5.36$\pm$ 5.66 \\

            Jacobi-ASVD &  4.50 $\pm$ 4.06 & 1.75 $\pm$ 2.46 & 6.21 $\pm$ 5.38 & 0.85 $\pm$ 0.63 \\

            GRQI &  1.61 $\pm$ 0.77 & 1.52 $\pm$ 0.95 & 1.27 $\pm$ 0.95 & 1.16 $\pm$ 0.73 \\
            \midrule
            HOSCF &  8.12 $\pm$ 0.52 & 9.38 $\pm$ 0.54 & 11.23 $\pm$ 0.33 & 12.43 $\pm$ 0.40 \\
            iHOSCF & 8.25 $\pm$ 0.50 & 9.58 $\pm$ 0.57 & 11.25 $\pm$ 0.34 & 12.46 $\pm$ 0.38 \\
			\bottomrule
		\end{tabular}
	\end{center}
\end{table}

Figure \ref{fig:iterations} shows that the decoupling algorithms, Jacobi-ASVD and GRQI, do not converge in most cases, and then from Table \ref{table:lambda}, we observe that they rarely obtain the value of $\lambda$ consistent with $\texttt{tucker\_als}$. For the Jacobi-HOPM algorithm, although the rank-one approximation of $\tensor{A}$ consistent with $\texttt{tucker\_als}$ can be obtained when $d=3$ and $4$, it requires more iterations than the proposed HOSCF algorithm. Specifically, the Jacobi-HOPM algorithm is $2.16\times$ and $2.87\times$ slower than the HOSCF algorithm for the cases of $d=3$ and $4$, respectively. As the order of $\tensor{A}$ increases, Figure \ref{fig:iterations} shows that the convergence of the Jacobi-HOPM algorithm is broken, then the obtained value of $\lambda$ shown in Table \ref{table:lambda} is incorrect.
And as we can see in Table \ref{table:lambda} and Figure \ref{fig:iterations}, for tensors of different orders, HOSCF and iHOSCF algorithms are guaranteed to converge to the rank-one approximation of $\tensor{A}$ consistent with $\texttt{tucker\_als}$. From Figure \ref{fig:iterations}, we can also observe that the iHOSCF algorithm reduces the number of iterations by approximately half compared to the HOSCF algorithm. Moreover, it is worth mentioning that the convergence speed of the HOSCF decreases with the growth of the order of $\tensor{A}$, which matches the convergence theory established in Section \ref{sec:section5}.

\begin{figure}
	\centering
	\includegraphics[width=1.00\textwidth,height=0.50\textwidth]{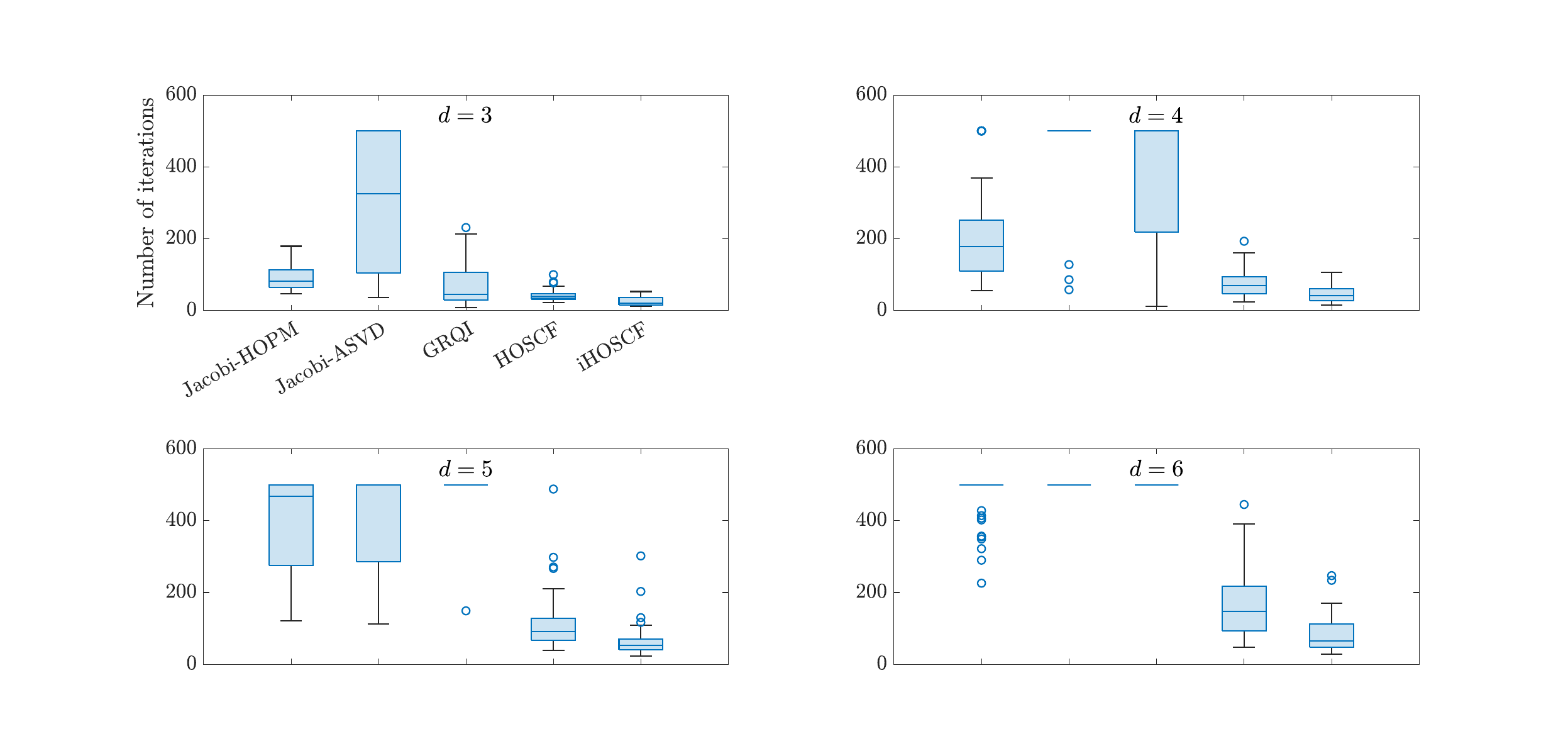}
	\caption{Number of iterations by all tested algorithms for Example 1.}\label{fig:iterations}
\end{figure}


\subsubsection{Example 2}\label{subsec:12}

In this example, we test the convergence of the proposed HOSCF and iHOSCF algorithms using three tensors from \cite{nie2014semidefinite}, which include third-, fourth-, and fifth-order tensors as listed in Table \ref{table:tensors}. The accuracy of the algorithm is measured by the best rank-one approximation ratio of the tensor $\bm{\mathcal{A}}$, defined in \cite{qi2011best} as $\rho(\bm{\mathcal{A}})=|\lambda|/{\|\bm{\mathcal{A}}\|_{F}}$, and the reference values of $\rho(\bm{\mathcal{A}})$ obtained by the HOPM algorithm are also shown in Table \ref{table:tensors}.
Table \ref{table:ex2} reports the average value of $\rho(\bm{\mathcal{A}})$ and the average number of iterations for all tested algorithms with 50 initial guesses.

\begin{table}
	\small
 \renewcommand\arraystretch{1.0}
 \setlength{\tabcolsep}{1.0mm}{
	\begin{center}
		\caption{\normalsize Summary of tensors used for Example 2.}\label{table:tensors}
		\begin{tabular}{c|c|c|c|c}
			\toprule
  Name & Order & Dimension & $\bm{\mathcal{A}}_{i_{1},i_{2},\cdots,i_{d}}$ & $\rho(\bm{\mathcal{A}})$ \\
  \midrule
EXP & $d=3$ &  $30\times30\times30$ & $\sum\limits_{j=1}^{d}(-1)^{j+1}j\mathrm{exp}(-i_{j})$ & 0.82$\pm$0.00 \\
\midrule
ARCSIN & $d=4$ &  $20\times20\times20\times20$& $\left\{\begin{array}{cc}
       \sum\limits_{j=1}^{d}\arcsin\left((-1)^{i_{j}\frac{j}{i_{j}}}\right) & \text{for all}\ i_{j}\geq j \\
       0 & \text{otherwise}
   \end{array}\right.$ & 0.66$\pm0.01$
  \\
  \midrule
TAN & $d=5$ & $10\times10\times10\times10\times10$   & $\tan\left(\sum\limits_{j=1}^{d}(-1)^{j+1}\frac{i_{j}}{j}\right)$ & $0.27\pm0.08$ \\
			\bottomrule
		\end{tabular}
	\end{center}}
\end{table}

From Table \ref{table:ex2}, it can be observed that all tested algorithms, except the GRQI algorithm, can obtain the exact value of $\rho(\bm{\mathcal{A}})$ for the tensor EXP, and our proposed algorithms converges more than $2\times$ faster than Jacobi-HOPM and -ASVD. For the tensor ARCSIN, only the Jacobi-HOPM and our proposed algorithms obtain the same value of $\rho(\bm{\mathcal{A}})$ as $\texttt{tucker\_als}$, but the convergence speed of the HOSCF algorithm is much faster than Jacobi-HOPM, nearly $35\times$. For the tensor TAN, the values of $\rho(\bm{\mathcal{A}})$ obtained by all comparison algorithms are incorrect, while HOSCF and iHOSCF can still find the best rank-one approximation of $\bm{\mathcal{A}}$ with a faster convergence speed. Moreover, we also observe that the iHOSCF algorithm is $1.52\times\sim3.48\times$ faster than the HOSCF algorithm, and is more stable, especially for the tensor TAN.

\begin{table}
	\small
 \renewcommand\arraystretch{1.0}
 \setlength{\tabcolsep}{1.0mm}{
	\begin{center}
		\caption{\normalsize Values of $\rho(\bm{\mathcal{A}})$ and the number of iterations by all tested algorithms for Example 2.}\label{table:ex2}
		\begin{tabular}{c|c|c|c|c|c|c}
			\toprule
			\multirow{2}{*}{Algorithms} & \multicolumn{2}{c|}{EXP} & \multicolumn{2}{c|}{ARCSIN} & \multicolumn{2}{c}{TAN}\\
			\cline{2-7}
            & $\rho(\bm{\mathcal{A}})$ & $\#$Iterations & $\rho(\bm{\mathcal{A}})$ & $\#$Iterations & $\rho(\bm{\mathcal{A}})$ & $\#$Iterations\\
            \midrule

            Jacobi-HOPM &  0.82$\pm0.00$ & 26.00$\pm$3.60 & 0.66$\pm0.00$ & 498.46$\pm$6.80 & 0.14$\pm0.14$ & $378.50\pm151.67$ \\

            Jacobi-ASVD &  0.82$\pm0.00$ & 24.14$\pm$2.96 & 0.43$\pm$0.21 & 275.84$\pm$245.34 & 0.08$\pm$0.12 &$409.10\pm183.67$ \\

            GRQI &  0.57$\pm$0.23 & 6.04$\pm$1.14 & 0.56$\pm$0.10 & 6.56$\pm$1.18 & 0.01$\pm$0.01&$464.24\pm104.71$ \\
            \midrule
            HOSCF &  0.82$\pm0.00$ & 11.68$\pm$1.85& 0.66$\pm0.01$ & 14.26$\pm$1.05 & 0.28$\pm0.06$ & 66.84$\pm114.47$ \\
           iHOSCF & 0.82$\pm0.00$ & 7.42$\pm$1.16 & 0.66$\pm0.01$ & 9.38$\pm$1.07 & $0.26\pm0.07$ &$19.22\pm8.22$ \\
			\bottomrule
		\end{tabular}
	\end{center}}
\end{table}

\subsubsection{Example 3}
The goal of this example is to apply HOSCF to the greedy rank-one update algorithm for testing its convergence. The input tensor $\tensor{A}\in\mathbb{R}^{784\times5000\times10}$ is composed of the training dataset
from the MINST database \cite{MNIST}, where the first mode is the texel mode, the second mode corresponds to training images, and the third mode represents image categories.
We set the CP-rank $R$ to $5$, and let $\tensor{B}^{(r)}$ be the best rank-one approximation of $\tensor{A}^{(r)}$, where $\tensor{A}^{(r+1)} = \tensor{A}^{(r)}-\tensor{B}^{(r)}$ and $\tensor{A}^{(1)}=\tensor{A}$. The average computed residuals $\text{res}_{r}=\frac{\|\tensor{A}^{(r)}\|_F}{\|\tensor{A}^{(1)}\|_F}$ and the average number of iterations obtained by the tested algorithms with 50 initial guesses are reported in Table \ref{table:ex3-2}, and Figure \ref{fig:ex3-2} shows the first, third, and fifth texel factors.

From Table \ref{table:ex3-2} and Figure \ref{fig:ex3-2}, it can be seen that the convergence of GRQI is broken as the CP-rank increases.
Although Jacobi-HOPM and -ASVD yield similar residuals to $\texttt{tucker\_als}$, they require $1.25\times\sim5.31\times$ and $1.24\times\sim13.86\times$ more iterations than the proposed HOSCF and iHOSCF, respectively. It is worth mentioning that the residual obtained by HOSCF is the most accurate of all tested algorithms, its corresponding texel factors are shown in Figure \ref{fig:ex3-2} are consistent with $\texttt{tucker\_als}$.

\begin{table}
	\small
 \renewcommand\arraystretch{1.0}
 \setlength{\tabcolsep}{0.85mm}{
	\begin{center}
		\caption{\normalsize The computed residual and the number of iterations by all tested algorithms for Example 3.}\label{table:ex3-2}
		\begin{tabular}{c|c|c|c|c|c|c}
			\toprule
			\multirow{2}{*}{Algorithms} & \multicolumn{2}{c|}{$r=1$} & \multicolumn{2}{c|}{$r=3$} & \multicolumn{2}{c}{$r=5$}\\
			\cline{2-7}
            & $\text{res}_{r+1}$ & $\#$Iterations & $\text{res}_{r+1}$ & $\#$Iterations & $\text{res}_{r+1}$ & $\#$Iterations\\
            \midrule
            \texttt{tucker\_als} &  0.7663 & -- & 0.7247 & -- & 0.7006 & -- \\
            \midrule
            Jacobi-HOPM &  0.7663 & 6.32$\pm$0.47 & 0.7252 & 75.36$\pm$16.50 & 0.7009 & $92.92\pm37.84$ \\

            Jacobi-ASVD &  0.7663 & 6.26$\pm$0.44 & 0.7280 & 153.02$\pm$164.59 & 0.7095 &$242.28\pm195.45$ \\

            GRQI &  0.7663 & 3.98$\pm$0.14 & 0.7663 & 357.34$\pm$177.72 & 0.7663&$344.06\pm172.94$ \\
            \midrule
            HOSCF &  0.7663 & 5.04$\pm$0.20& 0.7247 & 19.58$\pm$3.57 & 0.7006 & 24.78$\pm6.15$ \\
           iHOSCF & 0.7663 & 4.00$\pm$0.00 & 0.7255 & 14.34$\pm$5.08 & 0.7013 &$17.48\pm7.10$ \\
			\bottomrule
		\end{tabular}
	\end{center}}
\end{table}

\begin{figure}
	\centering
	\includegraphics[width=1.0\hsize]{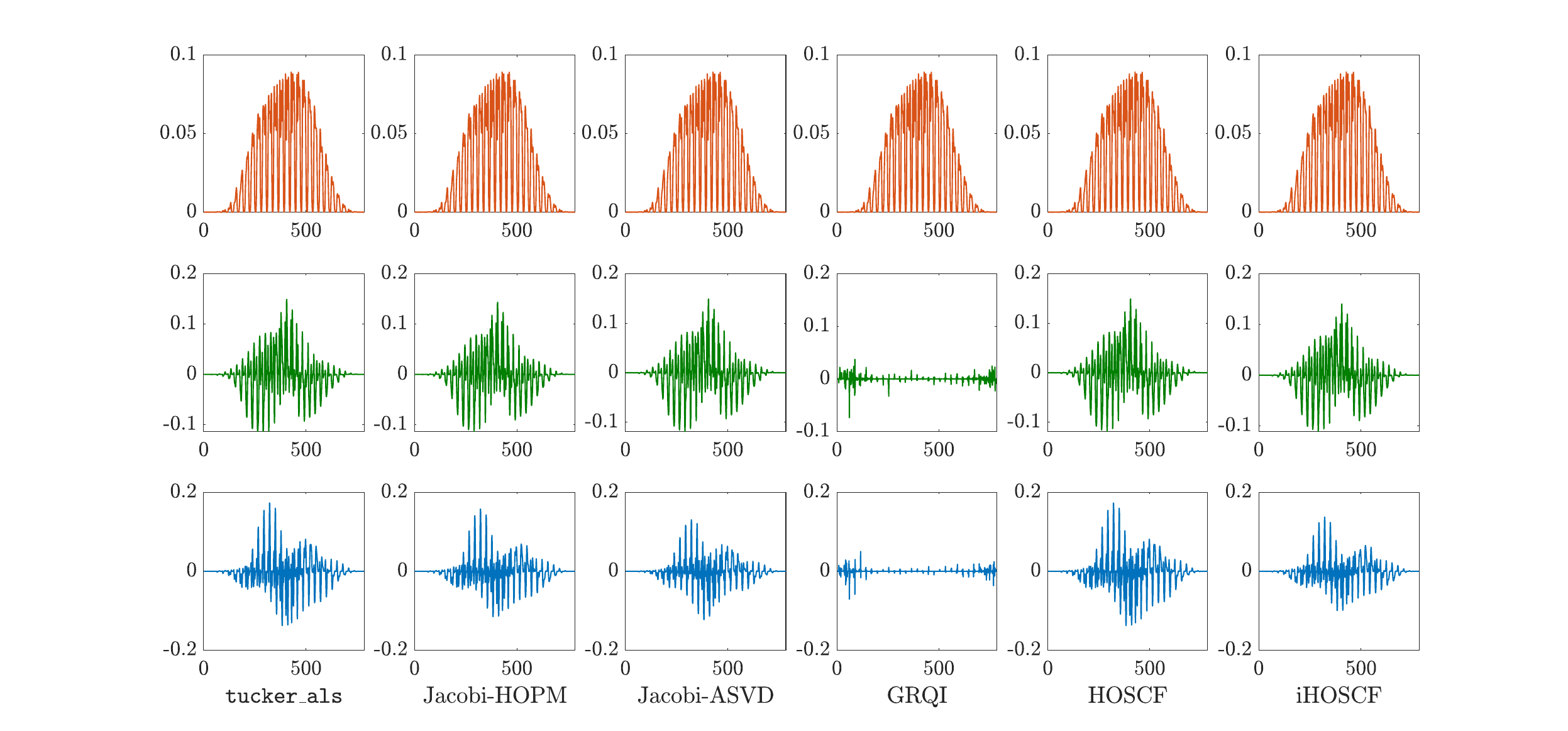}
	\caption{The first, third, and fifth texel factors by all tested algorithms for Example 3.}\label{fig:ex3-2}
\end{figure}

\subsection{Parallel scalability}

In the second part of the numerical experiments, we test the parallel scalability of the proposed HOSCF algorithm on several tensors of different orders and compare it with the most popular coupling algorithm, i.e., the HOPM algorithm. To ensure these input tensors generated randomly have the same scale ($16$ M), we set the sizes of the input tensors to ${32\times32\times32\times32\times16}$, ${16\times16\times16\times16\times16\times16}$, ${16\times16\times16\times16\times8\times8\times8}$, and $8\times8\times8\times8\times8\times8\times8\times8$, respectively.
The number of iterations is set to $10$ for both HOPM and HOSCF, and the recorded running time is the average value obtained after 10 repetitions.

We break down the running time of the HOSCF algorithm into three parts,
i.e., the construction of the intermediate matrix $\bm{J}(\bm{x}_{k})$, the calculation of the largest magnitude eigenpair of $\bm{J}(\bm{x}_{k})$, and others.
Table \ref{table:ex3} shows the time cost and proportion of each part in the HOSCF algorithm on a single processor core. {From this table, we can observe that the time cost to construct the intermediate matrix $\bm{J}(\bm{x}_{k})$ for tensors of different orders account for more than $99\%$, which means that it is the most time-consuming part of the HOSCF algorithm, matches the computational complexity analysis in Section \ref{sec:section4}.} Theoretically, the HOSCF algorithm is highly scalable on multi-core computers due to the high parallelism of the TTVc operations involved in the construction of $\bm{J}(\bm{x}_k)$.

\begin{table}
\renewcommand\arraystretch{1.0}
	\setlength{\tabcolsep}{1.0mm}{
	\begin{center}
		\caption{\normalsize Running time (s) and proportion ($\%$) of each part of the HOSCF algorithm in the case of a single processor core.}\label{table:ex3}
		\begin{tabular}{c|c|c|c|c}
			\toprule
			\multirow{2}{*}{Order of tensors} & \multicolumn{2}{c|}{Intermediate matrix $\bm{J}(\bm{x}_{k})$} & \multicolumn{2}{c}{Largest magnitude eigenpair of $\bm{J}(\bm{x}_{k})$} \\
   \cline{2-5}
  &  Running time & Proportion  & Running time  & Proportion \\
  \midrule
$d=5$ & 7.18 & 99.54 & 0.021 & 0.29 \\
$d=6$ & 10.07 & 99.77 & 0.011 & 0.11 \\
$d=7$ & 12.30 & 99.80 & 0.008 & 0.07 \\
$d=8$ & 15.82 & 99.85 & 0.006 & 0.04 \\
			\bottomrule
		\end{tabular}
	\end{center}}
\end{table}

\begin{figure}
	\centering	\includegraphics[width=1.00\textwidth,height=0.50\textwidth]{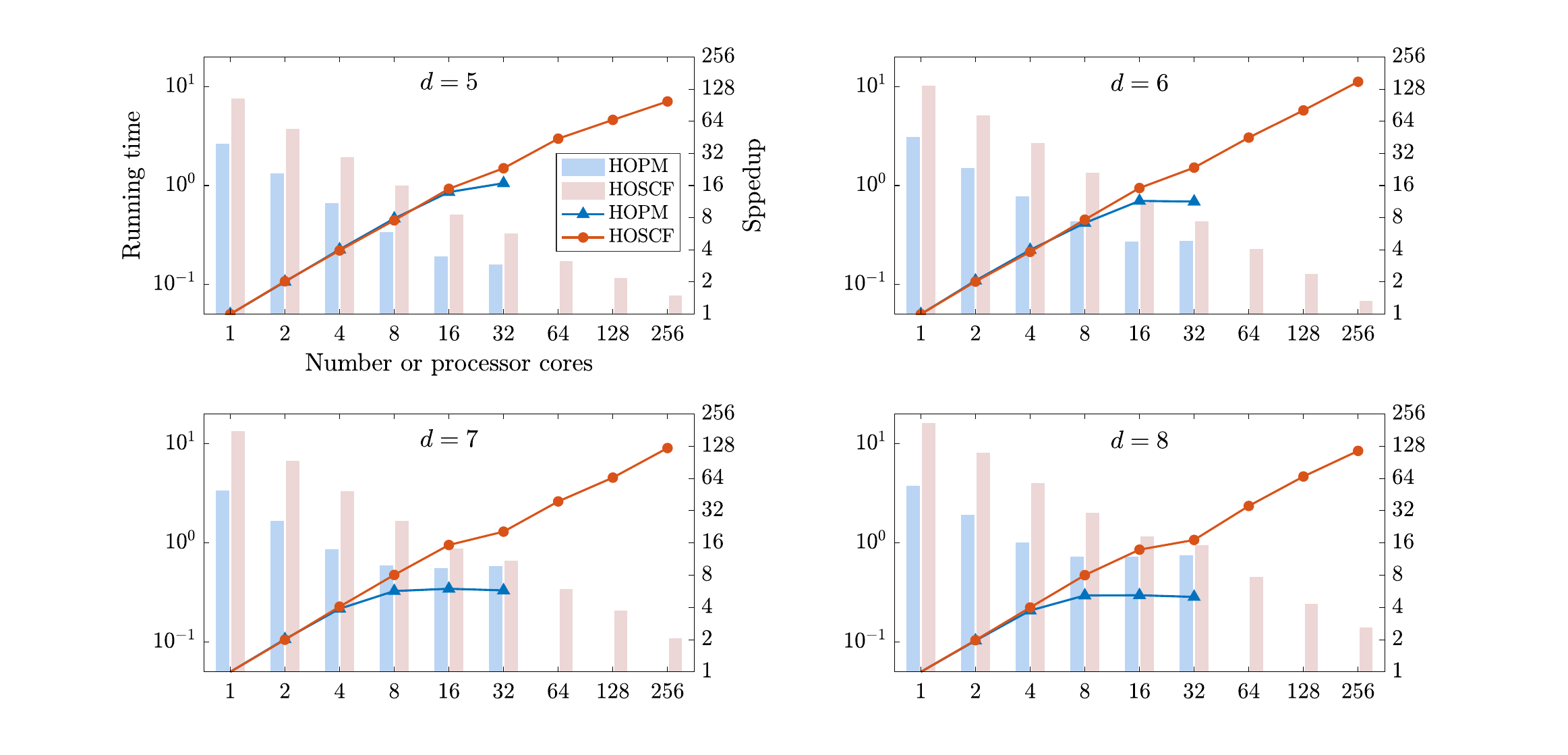}
	\caption{Running time (s) and speedup of the HOPM and HOSCF algorithms for tensors of different orders on a multi-core computer.}\label{fig:scability}
\end{figure}

Figure \ref{fig:scability} illustrates the running time and speedup of the HOPM and HOSCF algorithms as the number of processor cores increases from 1 to 256 (on 8 compute nodes). Since the TTVc operations in the HOPM algorithm are interdependent and lack enough parallelism, it cannot be further accelerated by increasing compute nodes, thus only the speedup of the HOPM algorithm with 1 to 32 cores on a single compute node is shown in the figure. 
Although the HOSCF algorithm is slower than HOPM when the number of processor cores is small, the parallel scalability is much better. And more importantly, HOSCF can scale across multiple compute nodes where HOPM cannot.
Indeed, even on a single compute node, the scalability of the HOPM algorithm is poor, especially for higher-order tensors.
In contrast, the parallelism of the HOSCF algorithm is not only evident in a single TTVc operation but also in the independence of these TTVc operations from each other. Therefore, the HOSCF algorithm has good parallel scalability on multi-core computers. Specifically in this test, when the number of processor cores is increased from 1 to 256, the HOSCF algorithm can achieve speedups of $98.03\times$, $149.90\times$, $123.02\times$, and $115.73\times$ for tensors with the order $d=5,6,7,8$, respectively.

\section{Conclusions}\label{sec:section7}
In this paper, we first pointed out that finding the best rank-one approximation of tensors can be reformulated as a special case of NEPv, and developed an efficient decoupling algorithm inspired by the SCF iteration to solve it. The established convergence theory illustrates that the proposed HOSCF algorithm is local $q$-linearly convergent, and we also provide an estimation of the convergence rate.
Further, combined with Rayleigh quotient iteration, we proposed an improved HOSCF algorithm, which significantly reduces the number of iterations of the HOSCF algorithm.
Compared with existing decoupling algorithms, numerical experiments show that the proposed algorithms can converge to the best rank-one approximation for both synthetic and real-world tensors. By implementing the HOSCF algorithm on a modern parallel computer, we observe that it has good parallel scalability and thus achieves high performance, which exhibits the advantage of decoupling algorithms. Possible future work may include conducting a more in-depth study on the convergence of HOSCF and developing more efficient acceleration techniques to further improve the convergence speed.

\section*{Acknowledgments}
C. Xiao has been supported by the Hunan Provincial Natural Science Foundation grant 2023JJ40005 and National Natural Science Foundation of China grant 12131002. C. Yang has been supported by National Natural Science Foundation of China grant 12131002.



\bibliographystyle{elsarticle-num}


\bibliography{hoscf.bib}




\end{document}